\documentclass[12pt,a4paper,english,reqno]{amsart}
\usepackage[a4paper,footskip=1.5em]{geometry}
\usepackage{amsmath,amssymb,amsthm,mathtools}
\usepackage[mathscr]{euscript}
\usepackage[usenames,dvipsnames]{color}
\usepackage{adjustbox,tikz,calc,graphics,babel,standalone}
\usepackage{subcaption}
\usepackage{csquotes,enumerate,verbatim}
\usepackage[final]{microtype}
\usepackage[numbers]{natbib}
\usetikzlibrary{shapes.misc,calc,intersections,patterns,decorations.pathreplacing}
\usepackage{hyperref}
\hypersetup{colorlinks=true,linkcolor=blue,citecolor=blue,pdfpagemode=UseNone,pdfstartview={XYZ null null 1.00}}
\usepackage{cmtiup}

\theoremstyle{plain}
\newtheorem*{theorem*}{Theorem}
\newtheorem{theorem}{Theorem}[section]

\newtheorem{claim}[theorem]{Claim}

\newtheorem*{claim*}{Claim}

\theoremstyle{remark}

\newcommand{\N}{\mathbb{N}}
\newcommand{\E}{\mathbb{E}}

\let\emptyset\varnothing

\let\originalleft\left
\let\originalright\right
\renewcommand{\left}{\mathopen{}\mathclose\bgroup\originalleft}
\renewcommand{\right}{\aftergroup\egroup\originalright}

\begin{document}

\title{An improved lower bound for Folkman's theorem}

\author{J\'{o}zsef Balogh}
\address{Department of Mathematics, University of Illinois, 1409 W.\/ Green Street, Urbana IL 61801, USA}
\email{jobal@math.uiuc.edu}

\author{Sean Eberhard}
\address{London NW5\thinspace3LT, UK}
\email{eberhard.math@gmail.com}

\author{Bhargav Narayanan}
\address{Department of Pure Mathematics and Mathematical Statistics, University of Cambridge, Wilberforce Road, Cambridge CB3\thinspace0WB, UK}
\email{b.p.narayanan@dpmms.cam.ac.uk}

\author{Andrew Treglown}
\address{School of Mathematics,	University of Birmingham, Edgbaston, Birmingham, B15\thinspace2TT, UK}
\email{a.c.treglown@bham.ac.uk}

\author{Adam Zsolt Wagner}
\address{Department of Mathematics, University of Illinois, 1409 W.\/ Green Street, Urbana IL 61801, USA}
\email{zawagne2@illinois.edu}

\date{27 February 2017}

\subjclass[2010]{Primary 05D10; Secondary 05D40}

\begin{abstract}
Folkman's theorem asserts that for each $k \in \N$, there exists a natural number $n = F(k)$ such that whenever the elements of $[n]$ are two-coloured, there exists a set $A \subset [n]$  of size $k$ with the property that all the sums of the form $\sum_{x \in B} x$, where $B$ is a nonempty subset of $A$, are contained in $[n]$ and have the same colour. In 1989, Erd\H{o}s and Spencer showed that $F(k) \ge 2^{ck^2/ \log k}$, where $c >0$ is an absolute constant; here, we improve this bound significantly by showing that $F(k) \ge 2^{2^{k-1}/k}$ for all $k\in \N$. 
\end{abstract}
\maketitle

\section{Introduction}
Schur's theorem, proved in 1916, is one of the central results of Ramsey theory and asserts that whenever the elements of $\N$ are finitely coloured, there exists a monochromatic set of the form $\{x, y, x+y\}$ for some $x , y \in \N$. About fifty years ago, a wide generalisation of Schur's theorem was obtained independently by Folkman, Rado and Sanders, and this generalisation is now commonly referred to as Folkman's theorem (see~\citep{graham}, for example). To state Folkman's theorem, it will be convenient to have some notation. For $n \in \N$, we write $[n]$ for the set $\{1, 2, \dots, n\}$, and for a finite set $A \subset \N$, let 
\[S(A) = \left \{ \sum_{x \in B} x : B \subset A \text{ and } B \ne \emptyset \right \}\]
denote the set of all \emph{finite sums} of $A$. In this language, Folkman's theorem states that for all $k,r\in\mathbb{N}$, there exists a natural number $n = F(k,r)$ such that whenever the elements of $[n]$ are $r$-coloured, there exists a set $A\subset [n]$ of size $k$ such that $S(A)$ is a monochromatic subset of $[n]$; of course, it is easy to see that Folkman's theorem, in the case where $k = 2$, implies Schur's theorem.

In this note, we shall be concerned with lower bounds for the two-colour Folkman numbers, i.e., for the quantity $F(k) = F(k,2)$. In 1989, Erd\H{o}s and Spencer~\citep{erdosspencer} proved that 
\begin{equation}\label{es}
F(k)\geq 2^{ck^2/\log k}
\end{equation} 
for all $k \in \N$, where $c>0$ is an absolute constant; here, and in what follows, all logarithms are to the base $2$. Our primary aim in this note is to improve~\eqref{es}.

Before we state and prove our main result, let us say a few words about the proof of~\eqref{es}. Erd\H{o}s and Spencer establish~\eqref{es} by considering uniformly random two-colourings. In particular, they show that if $[n]$ is two-coloured uniformly at random and additionally $n \le 2^{ck^2/\log k}$ for some suitably small absolute constant $c>0$, then with high probability, there is no $k$-set $A \subset [n]$ for which $S(A)$ is monochromatic. On the other hand, it is not hard to check that if $n \ge 2^{Ck^2}$ for some suitably large absolute constant $C>0$, then a two-colouring of $[n]$ chosen uniformly at random is such that, with high probability, there exists a set $A \subset [n]$ of size $k$ for which $S(A)$ is monochromatic; indeed, to see this, it is sufficient to restrict our attention to sets of the form $\{p, 2p, \dots, kp\}$, where $p$ is a prime in the interval $[n/\log^2 n, 2n/ \log^2 n]$, and notice that the sets of finite sums of such sets all have size $k(k+1)/2$ and are pairwise disjoint. With perhaps this fact in mind, in their paper, Erd\H{o}s and Spencer also describe some of their attempts at removing the factor of $\log k$ in the exponent in~\eqref{es}; nevertheless, their bound has not been improved upon since. 

Our main contribution is a new, doubly exponential, lower bound for $F(k)$, significantly strengthening the bound due to Erd\H{o}s and Spencer. 

\begin{theorem}\label{mainthm}
For all  $k \in \N$, we have 
\begin{equation}\label{nb}
F(k)\geq 2^{2^{k-1}/k}.
\end{equation}
\end{theorem}

This short note is organised as follows. We give the proof of Theorem~\ref{mainthm} in Section~\ref{pf} and conclude with some remarks in Section~\ref{conc}.

\section{Proof of the main result}\label{pf}
In this section, we give the proof of Theorem~\ref{mainthm}.
\begin{proof}[Proof of Theorem~\ref{mainthm}]
The result is easily verified when $k \le 3$, so suppose that $k\geq 4$  and let $n = \lfloor 2^{2^{k-1}/k} \rfloor$. In the light of our earlier remarks, a uniformly random colouring of $[n]$ is a poor candidate for establishing~\eqref{nb}. Instead, we generate a (random) red-blue colouring of $[n]$ as follows: we first red-blue colour the odd elements of $[n]$ uniformly at random, and then extend this colouring uniquely to all of $[n]$ by insisting that the colour of $2x$ be different from the colour of $x$ for each $x \in [n]$; hence, for example, if $5$ is initially coloured blue, then $10$ gets coloured red, $20$ gets coloured blue, and so on.
	
Fix a set $A \subset [n]$ of size $k$ with $S(A)\subset [n]$. We have the following estimate for the probability that $S(A)$ is monochromatic in our colouring.

\begin{claim} $\mathbb{P}(S(A) \text{ is monochromatic}) \le 2^{1-2^{k-1}}$.
\end{claim}
\begin{proof}
First, if $|S(A)|\leq 2^k-2$, then it is easy to see from the pigeonhole principle that there exist two subsets $B_1,B_2\subset A$ such that $\sum_{x\in B_1}x = \sum_{x\in B_2}x$, and by removing $B_1 \cap B_2$ from both $B_1$ and $B_2$ if necessary, these sets may further be assumed to be disjoint; in particular, this implies that $S(A)$ contains two elements one of which is twice the other. It therefore follows from the definition of our colouring that $S(A)$ cannot be monochromatic unless $|S(A)| = 2^k-1$.

Next, suppose that $|S(A)|=2^k - 1$. For each odd integer $m \in \N$, we define $G_m=\{m,2m,4m,\dots\}\cap [n]$, and note that these geometric progressions partition $[n]$. Observe that $S(A)$ intersects at least $2^{k-1}$ of these progressions. Indeed, if there is an odd integer $r\in A$ for example, then $S(A)$ contains exactly $2^{k-1}$ distinct odd elements and these elements must lie in different progressions. More generally, if each element of $A$ is divisible by $2^s$ and $s$ is maximal, then there exists an element $r$ of $A$ with $r=2^st$, where $t$ is odd; it is then clear that precisely $2^{k-1}$ elements of $S(A)$ are divisible by $2^{s}$ but not by $2^{s+1}$ and these elements must necessarily lie in different progressions. With this in mind, we define $B_A$ to be a maximal subset of $S(A)$ with the property $|B_A \cap G_m|\leq 1$ for each $m$; for example, we may take $B_A$ to consist of the least elements (where they exist) of the sets $S(A)\cap G_m$. Clearly, our red-blue colouring restricted to $B_A$ is a uniformly random colouring, so the probability that $B_A$ is monochromatic is $2^{1-|B_A|}$; it follows that the probability that $S(A)$ is monochromatic is at most $2^{1-|B_A|} \le 2^{1-2^{k-1}}$.
\end{proof}
It is now easy to see that if $X$ is the number of sets $A \subset [n]$ of size $k$ for which $S(A)$ is a monochromatic subset of $[n]$ in our colouring, then
\[
\E [X] \leq \binom{n}{k} 2^{1-2^{k-1}}\leq\left(\frac{en}{k}\right)^k 2^{1-2^{k-1}} 
\le \left( \frac{e2^{2^{k-1}/k}}{k} \right)^k\left(2^{1-2^{k-1}}\right)= 2\left(\frac{e}{k}\right)^k<1,
\]
where the last inequality holds for all $k \ge 4$. Hence, there exists a red-blue colouring of $[n]$ without any set $A$ of size $k$ for which $S(A)$ is a monochromatic subset of $[n]$, proving the result.
\end{proof}

\section{Conclusion}\label{conc}
We conclude this note with two remarks. First, using the original arguments of Erd\H{os} and Spencer~\cite{erdosspencer} in conjunction with an inverse Littlewood--Offord theorem of Nguyen and Vu~\cite{nguyenvu}, it is possible to improve~\eqref{es} (up to removing the factor of $\log k$ in the exponent) by just considering uniformly random two-colourings. Second, we note that while~\eqref{nb} improves significantly on~\eqref{es}, this lower bound is still considerably far from the best upper bound for $F(k)$, which is of tower type; see~\citep{ub}, for instance.

\section*{Acknowledgements}
The first author was partially supported by NSF Grant DMS-1500121 and an Arnold O. Beckman Research Award (UIUC Campus Research Board 15006). The fourth author would like to acknowledge support from EPSRC grant EP/M016641/1.

Some of the research in this paper was carried out while the first author was a Visiting Fellow Commoner at Trinity College, Cambridge and the fourth and fifth authors were visiting the University of Cambridge; we are grateful for the hospitality of both the College and the University.

\bibliographystyle{amsplain}
\bibliography{folkman}

\end{document}